\documentclass[english,12pt,oneside]{amsproc}
\usepackage[english]{babel}
\usepackage{a4wide}
\usepackage{amsthm}
\usepackage{graphics}
\usepackage{amsfonts, amssymb, amscd, amsmath}
\usepackage{latexsym}
\usepackage[matrix,arrow,curve]{xy}
\usepackage{mathabx}
\usepackage{color}
\usepackage{pbox}
\usepackage{tikz}
\usetikzlibrary{matrix,decorations.pathreplacing,positioning}
\usepackage{hyperref}
\usepackage[normalem]{ulem}

 \DeclareMathOperator{\pt}{pt}

\DeclareMathOperator{\rk}{rk}

\DeclareMathOperator{\cat}{cat}
\DeclareMathOperator{\Vect}{Vect}
\DeclareMathOperator{\Mod}{Mod}

\DeclareMathOperator{\PrePos}{PrePos}
\DeclareMathOperator{\AlexTop}{AlexTop}
\DeclareMathOperator{\FinSets}{FinSets}
\DeclareMathOperator{\Diag}{Diag}
\DeclareMathOperator{\PreShvs}{PreShvs}
\DeclareMathOperator{\Shvs}{Shvs}
\DeclareMathOperator{\Ob}{Ob}
\DeclareMathOperator{\ord}{ord}
\DeclareMathOperator{\Sets}{Sets}
\DeclareMathOperator{\Top}{Top}
\DeclareMathOperator{\hocolim}{hocolim}
\DeclareMathOperator{\op}{op}
\DeclareMathOperator{\link}{link}
\DeclareMathOperator{\Th}{Th}

\DeclareMathOperator{\glue}{glue}
\DeclareMathOperator{\edge}{edge}
\DeclareMathOperator{\ver}{vert}

\newcommand{\simc}{\!\!\sim}

\newcommand{\ko}{\Bbbk}
\newcommand{\Zo}{\mathbb{Z}}

\newcommand{\ca}[1]{\mathcal{#1}}

\newcommand{\dd}{\partial}

\newcommand{\Vv}{\mathcal{V}}

\newcounter{stmcounter}[section]

\numberwithin{equation}{section}

\theoremstyle{plain}
\newtheorem{cor}[stmcounter]{Corollary}

\newtheorem{thm}[stmcounter]{Theorem}

\newtheorem{prop}[stmcounter]{Proposition}
\newtheorem{lem}[stmcounter]{Lemma}

\theoremstyle{definition}
\newtheorem{defin}[stmcounter]{Definition}

\theoremstyle{remark}

\newtheorem{rem}[stmcounter]{Remark}
\newtheorem{con}[stmcounter]{Construction}

\begin{document}

\title[Étale spaces and generalized inflations]{Homotopy types of finite étale spaces and generalized inflations}

\author{Anton Ayzenberg}
\address{Noeon Research}
\email{ayzenberga@gmail.com}

\author{Nadya Khoroshavkina}
\thanks{N.K. was supported by the Russian Science Foundation (grant No. 24-11-00366).}
\address{%
\parbox[t]{\linewidth}{\raggedright
National Research University \glqq Higher School of Economics\grqq,\\
International Laboratory of Cluster Geometry\vspace{1ex}
}}

\email{vkhoroshavkina@gmail.com}

%
%
\subjclass[2020]{Primary: 06A06, 55U10, 54B40, 55P10 Secondary: 18F70, 06A07, 55U05, 18A30, 55R55, 14F45, 18A10, 55P15, 05E45}

\keywords{sheaf on finite topology, flabby sheaf, étale space, poset fiber theorem, simplicial poset, inflation of simplicial complex, homotopy wedge decomposition, homotopy colimit}

\begin{abstract}
Inflation of a simplicial complex $K$ is a construction well known in combinatorial topology. It replaces each vertex $i$ of $K$ with a finite number $n_i$ of its copies, and each simplex $\{i_0,\ldots,i_k\}$ with $n_{i_0}n_{i_1}\cdots n_{i_k}$ many copies so that the collection of vertex-copies is spanned by a simplex in the inflation if and only if their originals were spanned by a simplex in the original complex. The celebrated poset fiber theorem of Bj\"{o}rner, Wachs, and Welker describes the homotopy type of such inflation in terms of homotopy types of $K$ and its links. In the current paper, we introduce more general inflations over simplicial posets: we replace each simplex with an arbitrary finite set of copies. The way how these sets patch together is specified by a commutative diagram, or, equivalently, a sheaf on the corresponding finite topology. The generalized inflation can be understood as étale space of such sheaf. We prove that, whenever this inflation sheaf is flabby, the poset fiber theorem still applies. We prove all results similar to those known for vertex inflations. We also cover the previous result of the first author about homotopy types of clique complexes of multigraphs.
\end{abstract}

\maketitle

\section{Introduction}\label{secMotivation}

Consider a simplicial map of two finite simplicial complexes $f\colon N\to K$. Assume $f$ is nondegenerate, meaning that it preserves dimensions of simplices, and surjective on simplices. How is the homotopy type of the base space $K$ connected to that of $N$?

This question is too vague hence meaningless. If the vertices of an $(n-1)$-dimensional simplicial complex $N$ admit a proper coloring in $n$ colors, such coloring provides a nondegenerate simplicial map $c\colon N\to \Delta^{n-1}$. First barycentric subdivision of any complex admits such a coloring, henceforth, for a surjective nondgenerate map $f\colon N\to K$ the covering space $N$ can have arbitrarily complex homotopy type, while the base space $K$ is contractible. It is not that difficult to construct examples the other way round: when $N$ is contractible while the base space $K$ is not.

Nondegenerate maps of simplicial complexes (or simplicial posets) are ubiquitous in mathematics. Branched coverings (in particular branched coverings in the sense of Dold--Smith~\cite{BuchVesnin}) provide important examples of nondegenerate simplicial maps. In~\cite{GugninFr1,GugninFr2,GugninBound}, Gugnin studied relation between cohomology rings of the base space and its branched covering space. Construction of an inflated complex introduced by Wachs~\cite{Wachs} provides a relatively simple example when homotopy type of the covering space is explicitly described using homotopy type of the base space and links of all its simplices. This technique was generalized in~\cite{AyzRukh} to study clique complexes of multigraphs and, in particular, complexes of tournaments of directed graphs introduced in topological data analysis to analyze brain functional networks~\cite{Govc, GLS}. Every clique complex of a multigraph can be mapped, by a nondegenerate simplicial map, to a clique complex of the corresponding simple graph. Homotopy wedge decomposition proved in~\cite{AyzRukh} may speed up homology (and persistent homology) computations when the data comes in the form of multigraphs.

In~\cite{AbrBrand,AbrParadox} Abramsky et.al. proposed to use nondegenerate maps of simplicial complexes as a tool to estimate contextuality and non-locality of quantum systems. In their approach, the simplices of the base complex $K$ indicate subsets of features that can be simultaneously measured in a quantum system, while the simplices of the covering complex $N$ indicate the potential results of such measurements. Such model allows to distinguish classical and quantum systems, as the inability to coherently measure all characteristics at once in a quantum system may be seen in topological structures of $K$, $N$ and their interaction.

In the current paper, we develop the approach to the study of nondegenerate simplicial maps based on sheaf theory. The approach is straightforward, however, it allows to prove, in a universal manner, a number of known results from combinatorial and homotopy topology. Instead of considering the simplicial map $f\colon N\to K$, we consider the $\Sets$-valued diagram $D_f$ on the base complex $K$ defined as follows.

\begin{con}\label{conDiagramFromMap}
For each (nonempty) simplex $I\in K$, the value $D_f(I)$ is the finite set $f^{-1}(I)$, the set of all simplices lying in the preimage of $I$ under $f$. These sets are naturally related with each other: if $I\subset J$, then we have a map
\[
\glue_{J,I}\colon D_f(J)\to D_f(I)
\]
sending each preimage of the bigger simplex $J$ to its unique face that is being mapped to the subsimplex $I$ by the map $f$. See Fig.~\ref{figExampleDiagram} for explanation.

Obviously, all the maps $\glue_{J,I}$ together form a commutative diagram, which is a functor $D_f\colon \cat(K\setminus\{\varnothing\})^{\op}\to\Sets$. Since every poset corresponds to an Alexandrov topology, and every diagram on a poset corresponds to a sheaf on this topology (as explained in subsections~\ref{AT} and~\ref{SAT}), we get a certain sheaf $\ca{D}_f$ on the Alexandrov topological space $X_{(K\setminus\{\varnothing\})^*}$ corresponding to the poset of nonempty simplices of $K$ with reversed order. The original covering complex $N$ can be recovered as the étale space of the sheaf $\ca{D}_f$.
\end{con}

\begin{figure}
  \centering
  \includegraphics[width=0.7\textwidth]{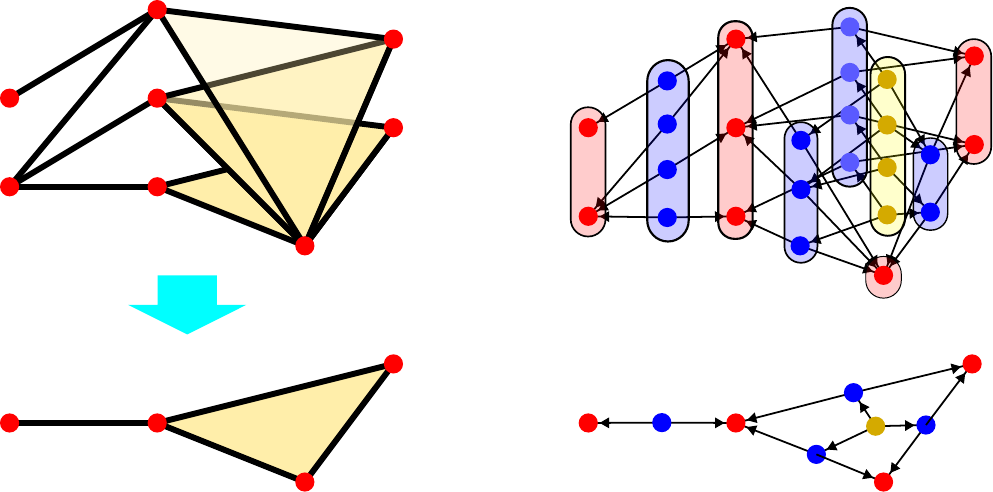}
  \caption{From simplicial map to $\Sets$-valued diagram over poset}\label{figExampleDiagram}
\end{figure}

The main conceptual contribution of our paper is the following claim: the properties of the sheaf $\ca{D}_f$ are useful in the study of homotopy types of covering spaces. We prove the following concrete result.

\begin{thm}\label{thmMainIntro}
Assume that $f\colon N\to K$ is a surjective nondegenerate map of simplicial complexes (or simplicial posets). Assume moreover, that the corresponding sheaf $\ca{D}_f$ is flabby.
\begin{enumerate}
  \item If $K=\Delta^{n-1}$ is an $(n-1)$-dimensional simplex, then $N$ is homotopy equivalent to a wedge of $(n-1)$-dimensional spheres.
  \item For general $K$, there is a homotopy wedge decomposition that expresses homotopy type of $N$ through homotopy types of $K$ and all its links (see formula~\eqref{eqMainWedgeDecomposition})
  \item If $K$ is homotopically Cohen--Macaulay complex of dimension $n-1$, then so is $N$.
\end{enumerate}
\end{thm}

This theorem generalizes the result of Wachs~\cite{Wachs} concerning vertex-inflated complexes, and all the generalizations concerning simplex-inflations obtained previously in~\cite{AyzRukh}. Recalling that classical measurement systems are ontologized in~\cite{AbrBrand} precisely by flabby sheaves, we get the following interesting corollary: classical systems, on the topological level, are expected to behave according to Theorem~\ref{thmMainIntro}, while any discrepancy (e.g. preimage of a $k$-dimensional simplex has nontrivial simplicial homology in degrees $<k$) is the evidence towards quantum behavior of the system. Notice that in the approach of the current work we never mention any homology or cohomology: all results are formulated and proved in non-Abelian setting.

One of our main approaches includes the notion of generalized inflation, which we would call inflation for simplicity. An arbitrary poset $S$ can be inflated along an arbitrary diagram $D$ on the dual poset $S^*$. Roughly, $S$ corresponds to the base space $K$ and its inflation corresponds to the covering space $N$.

The paper is organized as follows. In the remaining sections, until the last one, we forget about nondegenerate simplicial maps and work entirely with $\Sets$-valued diagrams over posets, sheaves over finite topologies, and their étale spaces. Section~\ref{secOverviewSheaves} provides an overview of basic constructions related to posets, finite topological spaces, sheaves, and diagrams. Section~\ref{secOverviewCombTop} provides an overview of basic constructions related to combinatorial and homotopy topology: definitions of simplicial posets, homotopically Cohen--Macaulay posets. It also describes the relation between étale spaces of sheaves and homotopy colimits of diagrams. In Section~\ref{secInflation} we introduce the notion of generalized inflation along a diagram and list its straightforward properties. We provide two motivating examples: vertex-inflations introduced by Wachs and clique complexes of multigraphs. In Section~\ref{secFormulations} we formulate three main technical theorems. In Section~\ref{secProofs} we prove all these theorems simultaneously by induction. In the final Section~\ref{secFinalRemarks} we derive Theorem~\ref{thmMainIntro} and explain how the results of papers~\cite{Wachs} and~\cite{AyzRukh} are obtained as particular cases of our constructions and claims.

\section{Overview: sheaves and finite topologies}\label{secOverviewSheaves}

Detailed exposition of majority of constructions related to sheaves over finite topologies can be found in the mathematical appendix of the preprint~\cite{AyzSheaves}.

\subsection{Preorders and Notation}

Let $S$ be a partially ordered set (poset). For each such set $S$ we can define a small category $\cat(S)$, whose objects are elements from $S$,  and  there is only one morphism from $s_1$ to $s_2$ in case $s_1\leqslant s_2$, and no morphisms otherwise.

A map between posets $f\colon S\to T$ is called monotonic (or morphism of posets) if the inequality $s_1\leqslant s_2$ in $S$ implies the inequality $f(s_1)\leqslant f(s_2)$ in $T$. Each such morphism induces a functor $\cat(f)\colon \cat(S)\to \cat(T)$ between the corresponding small categories. Similar definitions take place for partial preorders.

For any poset $S$  and  any element $s\in S$ we define $S_{\geqslant s}=\{t\in S\mid t\geqslant s\}$, and $S_{>s}=\{t\in S\mid t>s\}$. Posets $S_{\leqslant s}$, $S_{<s}$ are defined in a similar fashion.

Poset $S$ with the reversed order is denoted by $S^*$.

A map between posets $f\colon S\to T$ is called \emph{an embedding} if it is injective. A map $f$ is called an \emph{exact embedding} if, additionally to the injectivity, the inequality $f(s_1)\leqslant_T f(s_2)$ implies $s_1\leqslant_Ss_2$ (i.e. $S$ is a subset of $T$, and the order is inherited from $T$). Exact embeddings are denoted by $f\colon S\hookrightarrow T$.

\subsection{Alexandrov Topologies}\label{AT}

Let $X$ be a set, and $\Omega$ be a topology on $X$, that is a collection of subsets of $X$ called open subsets. Recall that a topological space $(X,\Omega)$ is called Alexandrov topological space if it satisfies the additional property that the intersection of any (possibly infinite) number of open sets is open. A topology $\Omega$ on $X$ is said to satisfy the separation axiom $T_0$ (Kolmogorov axiom) if, for each pair of points $x_1\neq x_2\in X$, either there is an open neighborhood $U$ of $x_1$, which does not contain $x_2$, or an open neighborhood $U$ of $x_2$, which does not contain $x_1$.

The following correspondence (cryptomorphism) between posets and topological spaces is well known, see, for example,~\cite{Arenas}.

\begin{prop}\label{propPosTop}
The category $\PrePos$ of partially preordered sets and morphisms is naturally equivalent to the category $\AlexTop$ of Alexandrov topological spaces and continuous functions. Under this equivalence, posets correspond to Alexandrov topologies satisfying the $T_0$-separation axiom.
\end{prop}

The following exposition requires a number of notions used to construct this correspondence. Let us briefly recall them.

\begin{con}\label{conPrePosToSpace}
Let $S$ be a set with a preorder $\leqslant$. Consider the topology $\Omega_S$ on $S$ that consists of all upper order ideals. In other words, $U\in \Omega_S$ if and only if conditions $x\in U$ and $y\geqslant x$ imply $y\in U$. Then $X_S=(S,\Omega_S)$ is an Alexandrov topological space. If $S$ is a poset, then $X_S$ satisfies $T_0$-separation axiom.
\end{con}

\begin{con}\label{conSpaceToPrepos}
Let $(X,\Omega)$ be an Alexandrov topological space. Consider a point $x\in X$ and the family $\mathcal{U}=\{U\in \Omega\mid x\in U\}$ of all open neighbourhoods of $x$. Their total intersection
\begin{equation}\label{eqDefUx}
U_x=\bigcap\nolimits_{U\in \mathcal{U}}U
\end{equation}
is an open set, according to the definition of Alexandrov topology. The set $U_x$ is called \emph{the minimal open neighborhood} of $x$. We define preorder on $X$ by setting $x_1\leqslant x_2$ if and only if $U_{x_1}\supseteq U_{x_2}$.
\end{con}

\begin{rem}\label{remMinNbhdIsACone}
In terms of the (pre)order on $S$, the minimal open neighbourhood $U_s$ of $s\in X_S$ in the topology $\Omega_S$ is just the upper cone over $s$, that is $U_s=S_{\geqslant s}$.
\end{rem}

\begin{rem}\label{remExactInclusion}
A morphism $f\colon S\to T$ between posets is an exact embedding if and only if the map $f\colon X_S\to X_T$ is an embedding and the topology on $X_S$ is induced from the topology on $X_T$. In such cases we denote the map by $f\colon X_S\hookrightarrow X_T$ similarly to posets.
\end{rem}

\subsection{Diagrams on partially ordered sets}

Let us fix a category $\Vv$. In homological algebra, a category $\Vv$ is usually the category of abelian groups  and  group homomorphisms, the category of rings  and  ring homomorphisms, or the category of vector spaces and linear maps. In our paper we basically restrict attention to the category $\Vv=\FinSets$ of finite sets and maps between them. Some basic definitions are given in greater generality for the completeness of exposition.

\begin{defin}
A $\Vv$-valued \emph{diagram on a poset} $S$ is a functor $D\colon \cat(S)\to \Vv$.
\end{defin}

All diagrams on $S$ form the category $\Diag(S,\Vv)$, whose morphisms are natural transformations of functors.

\subsection{Sheaves on topologies}

\begin{defin}\label{definPresheaf}
A \emph{presheaf} on a topological space $(X,\Omega_X)$ valued in a category $\Vv$ is a contravariant functor $\ca{F}\colon \cat(\Omega_X\setminus\{\varnothing\})^{op}\to \Vv$.
\end{defin}

Here the set $\Omega_X$ of open subsets is partially ordered by inclusion, $\cat(\Omega_X\setminus\{\varnothing\})^{op}$ is a category opposite to the category of this poset, that is the one obtained by reversing all arrows. In other words, presheaf can be considered as a rule that, to each non-empty open set $U\in\Omega_X$, assigns an object $\ca{F}(U)$ of $\Vv$, and, to each pair $V\subseteq U$, assigns a morphism
\[
r_{U\supseteq V}=\ca{F}(U\supseteq V)\colon \ca{F}(U)\to \ca{F}(V),
\]
called \emph{the restriction morphism}. All restriction morphisms satisfy consistency\footnote{Here and below we use the notation for the composition borrowed from programming and applied category theory: $f;g=g\circ f$.}: $r_{W\supseteq U};r_{U\supseteq V}=r_{W\supseteq V}$.

\begin{rem}
When $\Vv$ is the category $\Sets$ (or any concrete category, that is the one that behaves like ``sets with some structure''), an object $U\in V$ has elements.
In this case, the elements $u\in\ca{F}(U)$ are called \emph{sections of the sheaf} $\ca{F}$ on a subset $U$. If $U\supseteq V$ and $u\in\ca{F}(U)$, then $\ca{F}(U\supseteq V)(u)\in \ca{F}(V)$ is called \emph{the restriction of the section} $u$ to the subset $V$ and is denoted by $u|_V$.
\end{rem}

\begin{defin}\label{definInhabited}
A presheaf $\ca{F}$ with values in the category $\FinSets$ is called \emph{inhabited} if $\ca{F}(U)\neq\varnothing$ for any $U\in \Omega_X\setminus\{\varnothing\}$.
\end{defin}

\begin{defin}\label{definSheaf}
A presheaf $\ca{F}$ is called a \emph{sheaf} if the following two properties are satisfied:
\begin{enumerate}
  \item (Locality) Let $U$ be an open set, which is covered by a family of open subsets, $\{U_{i}\}_{i\in I}$, $U_{i}\subseteq U$ for each $i\in I$. Assume we are given two elements $u,v\in F(U)$ such that $u|_{U_{i}}=v|_{U_{i}}$ for each $i\in I$. Then the equality $u=v$ holds.
  \item (Gluing) Let $U$ be an open set  and  $\{U_{i}\}_{i\in I}$ be its open covering by subsets $U_{i}\subseteq U$ and let $\{u_{i}\in F(U_{i})\}_{i\in I}$ be a collection of sections. Suppose the sections agree on the overlap of their domains, that is $u_{i}|_{U_{i}\cap U_{j}}=u_{j}|_{U_{i}\cap U_{j}}$ for each $i,j\in I$. Then there exists a section $u\in F(U)$ such that $u|_{U_{i}}=u_{i}$ for each $i\in I$.
\end{enumerate}
\end{defin}

All presheaves on a topology $X$ valued in $\Vv$ form the category $\PreShvs(X,\Vv)$. The morphisms are natural transformations of functors from $\cat(\Omega_X\setminus\{\varnothing\})^{op}$ to $\Vv$. Sheaves form a complete subcategory $\Shvs(X,\Vv)$ in $\PreShvs(X,\Vv)$.

\begin{defin}\label{defStalk}
If $\ca{F}$ is a presheaf on a topological space $(X,\Omega_X)$ and $x\in X$, then
\[
\ca{F}_x=\lim\limits_{\substack{\longrightarrow\\x\in U, U\in\Omega_X}}\ca{F}(U)\hspace{6pt}\in \Ob(\Vv)
\]
(the direct limit of values of the presheaf over all open subsets which contain $x$), is called \emph{the stalk} of the presheaf $\ca{F}$ in $x$.
\end{defin}

This definition makes sense only if the category $\Vv$ is closed (that is contains all direct limits), or, if we speak about finite topologies, contains all finite limits. The category $\FinSets$ contains all finite limits, see~\cite{Maklein}. From the definition of a limit, it follows that for each $U\in \Omega_X$, $x\in U$, there exists a canonical morphism $\ca{F}(U\ni x)\colon \ca{F}(U)\to \ca{F}_x$.

\subsection{Sheaves on Alexandrov topologies}\label{SAT}

According to Proposition~\ref{propPosTop}, each poset $S$ is in a correspondence with a $T_0$-Alexandrov topology $X_S=(S,\Omega_S)$. There is a natural one-to-one correspondence between $\Vv$-valued diagrams on $S$ and $\Vv$-valued sheaves on $X_S$.

\begin{prop}\label{propDiagSheaf}
The category $\Diag(S,\Vv)$ of diagrams on a poset $S$ is equivalent to the category $\Shvs(X_S,\Vv)$ of sheaves on the corresponding Alexandrov topology.
\end{prop}

The claim is folklore, but the moderately detailed proof can be found in~\cite{AyzSheaves}. We adopt the following convention: the diagrams on a poset $S$ are denoted by straight letters and the corresponding sheaves on $X_S$ are denoted with the same letters written calligraphically. For example, a diagram $D$ corresponds to a sheaf $\ca{D}$. Proposition~\ref{propDiagSheaf} ensures that both $D$ and $\ca{D}$ carry the same information. However in some proofs it is important to distinguish them to maintain type consistency.

\begin{rem} \label{remConcreteLimSets}
The value of a sheaf on an open subset can be reconstructed from the stalks by the following formula

\begin{equation}\label{eqInvLimit}
\ca{D}(U)=\lim\limits_{\substack{\leftarrow \\ s\in U}}D(s)=\left\{(v_s\mid s\in U)\in\prod_{s\in U}D(s)\mid D(s_1\leqslant s_2)v_{s_1}=v_{s_2} \text{ for any }s_1\leqslant s_2 \right\}.
\end{equation}
In other words, a section $x\in \ca{D}(U)$ on an open set $U$ coincides with a collection of elements $D(s)$, for $s\in U$, compatible with respect to the maps of the diagram. In the presentation~\eqref{eqInvLimit} the restriction map $\ca{D}(U\supset V)$ is defined in a straightforward manner, because the family of compatible stalks on the set $U$ is also compatible on its subset $V$.
\end{rem}

\subsection{Trivial sheaf}

The following proposition is considered obvious in every book on algebraic geometry. It is indeed quite simple, but, at the same time, its explicit formulation comes in handy in further combinatorial applications.

\begin{prop}\label{propTrivialSheaf}
Let $\ca{D}$ be a $\Sets$-valued sheaf on the Alexandrov topology $X_S=(S,\Omega_S)$, which corresponds to the diagram $D$ on a poset $S$. The following two conditions are equivalent:
\begin{enumerate}
  \item All stalks $D(s)=\ca{D}(U_s)$, $s\in S$, are singletons.
  \item All values $\ca{D}(U)$, $U\in \Omega_S$, are singletons.
\end{enumerate}
\end{prop}

\begin{proof}
The implication $1\Rightarrow 2$ is trivial, since the minimal open neighbourhoods $U_s$ are open sets as well. The reverse implication follows from the explicit construction of the value $\ca{D}(U)$ out of the corresponding stalks, see Remark~\ref{remConcreteLimSets}. Indeed, the product of any number of singletons is a singleton,  and  the unique element of the product of the stalks satisfies the compatibility relations from the equation~\eqref{eqInvLimit}.
\end{proof}

\begin{defin}\label{defTrivialSheaf}
A sheaf $\ca{D}$ (and  the corresponding diagram $D$) with values in the category of sets is called trivial if it satisfies any of the equivalent conditions of Proposition~\ref{propTrivialSheaf}.
\end{defin}

Of course, Proposition~\ref{propTrivialSheaf} is valid for the topological spaces in general, not only for the Alexandrov topological spaces,  and  for an (almost) arbitrary category as a target category, instead of the category of sets: in that case we require all values (the corresponding stalks) to be final objects of the target category. In the abelian setting, it resembles the well-known fact that the module of global sections of the zero sheaf is zero\footnote{This situation should not be confused with the constant 1-dimensional sheaf in which case the module of global sections counts the number of connected components.}. An interested reader may solve this exercise to practise the machinery of abstract nonsense, but in our paper we don't need this level of generality.

\begin{lem}[Gluing Lemma]\label{lemGlue}
Let $\ca{D}$ be a sheaf on a topological space $X_S$. Let $U_1,U_2$ be two open sets in $X_S$, such that for each point $x\in U_1\cap U_2$ the stalk $D(x)$ is a singleton. Then, for any two sections $v_1\in \ca{D}(U_1)$, $v_2\in \ca{D}(U_2)$, there exists a unique section $v\in \ca{D}(U_1\cup U_2)$ which restricts to $v_1$ on $U_1$ and to $v_2$ on $U_2$.
\end{lem}

\begin{proof}
According to proposition~\ref{propTrivialSheaf}, we have $\ca{D}(U_1\cap U_2)=\{\ast\}$, a singleton. Therefore, $v_1|_{U_1\cap U_2}=v_2|_{U_1\cap U_2}$. Whereas the sections are compatible on the overlap, they extend to the common section, according to the gluing property from Definition~\ref{definSheaf}.
\end{proof}

\subsection{Induced sheaf, restriction of a sheaf}

Each continuous map of topological spaces $f\colon X\to Y$ induces direct  and  inverse image functors on categories of sheaves
\[
f_*\colon \Shvs(X;\Vv)\to \Shvs(Y;\Vv), \qquad f^*\colon \Shvs(Y;\Vv)\to \Shvs(X;\Vv).
\]
We will not need this construction in its full generality, but provide it for the case of posets. In this case the definition of the functor $f^*$ (inverse image of a sheaf) is significantly simplified at the conceptual level ---  and  as a result, it finds numerous applications in applied topology~\cite{Curry} and contemporary approaches to neural networks~\cite{ShDif, AyzSheaves}.

Let $f\colon S\to T$ be a morphism of posets (equivalently, a continious map between the corresponding Alexandrov topological spaces $f\colon X_S\to X_T$).

\begin{con}\label{conDirectImage}
Let $\ca{F}\in\Shvs(X_S;\Vv)$ be a sheaf on the Alexandrov topological space $X_S$. Consider a presheaf $f_*\ca{F}$ on the topological space $X_T$, whose values on open sets $U\in\Omega_T$  and inclusions $U\supseteq V$ are defined by
\[
f_*\ca{F}(U)=\ca{F}(f^{-1}(U)),\qquad f_*\ca{F}(U\supseteq V)=\ca{F}(f^{-1}(U)\supseteq f^{-1}(V)).
\]
Presheaf $f_*\ca{F}$ is known to be a sheaf~\cite{Iversen}. It is called \emph{the direct image} of $\ca{F}$ under the continuous map $f$.
\end{con}

We note that Construction~\ref{conDirectImage} is applicable without changes to continuous maps between any topological spaces. The map $f_*$ is functorial, which means that it is a correctly defined functor from the category $\Shvs(X_S;\Vv)$ to the category $\Shvs(X_T;\Vv)$.

The inverse image functor is more difficult to define, because an image of an open set may not be open under a continuous map. But in the case of posets  and  Alexandrov topological spaces, the construction of the inverse functor image is simplier due to the equivalence between a sheaf  and  the diagram of its stalks. Suppose again that $f\colon S\to T$ is a morphism of posets  and  $\ca{G}\in\Shvs(X_T;\Vv)$ is a sheaf on the topological space $X_T$. According to Proposition~\ref{propDiagSheaf}, it corresponds to a diagram $G\colon \cat(T)\to\Vv$.

\begin{con}\label{conInverseImage}
\emph{The inverse image} of the diagram $G$ on $T$ (or the sheaf $\ca{G}$ on $X_T$) under a morphism $f\colon S\to T$ is
a diagram $f^*G$ on the poset $S$ (or a sheaf $f^*\ca{G}$ on $X_S$), which is defined by the formulas:
\[
f^*G(s)=G(f(s)),\qquad f^*G(s\leqslant s')=G(f(s)\leqslant f(s')).
\]
\end{con}

This definition of inverse image in terms of diagrams easily implies that $f^*$ is functorial, that is it defines a functor $f^*\colon\Shvs(X_T;\Vv)\to \Shvs(X_S;\Vv)$. This functor is called \emph{the functor of inverse image} induced by the continuous map $f$.

\begin{rem}\label{remRestrictionOfSheaf}
A fairly important special case is the case when the map $f\colon S\hookrightarrow T$ is an exact embedding of posets (or, equivalently, the corresponding topological spaces, see Remark~\ref{remExactInclusion}). In this case, the set $S$ (to be precise, its image $A=f(S)$) is interpreted as a subset of $T$, and the sheaf $f^*\ca{G}$ is interpreted as a restriction of the sheaf $\ca{G}$ to the subset~$A$. It is natural to denote the inverse image $f^*\ca{G}$ as $\ca{G}|_A$ in this particular case.

If $A\subset T$ is an open subset in the topology of $X_T$, then we have $\ca{G}|_A(U)=\ca{G}(U)$ for each open subset $U\subseteq A$.
Notice that all open subsets of the (induced) topology on $X_A$ are also open in the topology of $X_T$.
\end{rem}

\subsection{Flabby sheaves}

\begin{defin}\label{defFlabby}
A $\FinSets$-valued sheaf $\ca{F}$ on a topological space $(X,\Omega_X)$ is called \emph{flabby} if, for each pair of nested open sets $U,V\in\Omega_X$, the restriction morphism $\ca{F}(U\supseteq V)$ is surjective.
\end{defin}

The definition is similar for sheaves valued in other concrete categories, such at $\Vect_\ko$ and $\Mod_R$, but we won't need this bigger generality. Informally, flabby sheaves are the sheaves for which any section on a smaller open set can be extended to a section on a bigger open set.

\begin{rem}\label{remFlabbyOnOpen}
If a map $S\hookrightarrow T$ is an exact embedding  and  $A=f(S)$ is open in the topological space $X_T$, then, from the flabbiness of $\ca{G}$ on $X_T$, it follows the flabbiness of its restriction $\ca{G}|_A=f^*\ca{G}$ on $A$, see Remark~\ref{remRestrictionOfSheaf}. Indeed, each restriction morphism in the sheaf $\ca{G}|_A$ is also a restriction morphism in the sheaf $\ca{G}$; if all the morphisms of the latter sheaf are surjective, then the same holds for any subset of them, in particular for the set of morphisms of the restricted sheaf.
\end{rem}

\section{Overview: combinatorial and homotopy topology}\label{secOverviewCombTop}

\subsection{Geometric realization of posets}

In combinatorial topology, there is a classical way to transform a poset into a simplicial complex.

\begin{defin}
Let $S$ be a poset. \emph{The order complex} of $S$ is a simplicial complex $\ord S$, whose vertex set coincides with $S$  and  whose simplices are chains  (well-ordered subsets) in $S$. The standard geometric realization $|\ord S|$ of the simplicial complex $\ord S$ is called \emph{the geometric realization} of $S$  and  is denoted by $|S|$.
\end{defin}

All posets in this paper are finite. Hence, their geometric realizations are finite CW-complexes.

\subsection{Homotopy colimits}

Let $D\colon \cat(S)\to \Vv$ be a $\Vv$-valued diagram on a poset, where $\Vv$ is either the category $\Sets$, or the category $\Top$ (for our main results we will use $\FinSets$ and assume all sets have discrete topology). We refer to~\cite{WZZ} for accessible exposition of homotopy colimits of diagrams over finite posets.

Let $s_1\leqslant s_2$ in $S$. There are two continuous maps: a map of a diagram $D(s_1\leqslant s_2)\colon D(s_1)\to D(s_2)$ and an embedding of cones $i(s_2\geqslant s_1)\colon |S_{\geqslant s_2}|\to |S_{\geqslant s_1}|$.

\begin{defin}
\emph{The homotopy colimit} of a diagram $D$ (on a poset $S$) is the topological quotient space
\[
\hocolim_{S}D=\left(\bigsqcup\nolimits_{s\in S}|S_{\geqslant s}|\times D(s)\right)/\simc,
\]
where the equivalence relation $\sim$ is defined by the condition: $(x,D(s_1\leqslant s_2)(y))\sim (i(s_2\geqslant s_1)(x),y)$ for any $s_1\leqslant s_2$, $x\in |S_{\geqslant s_2}|$, $y\in D(s_1)$.
\end{defin}

\begin{rem}
In more abstract terms, the homotopy colimit is the coend of the profunctor $\ca{H}\colon \cat(S)^{\op}\times\cat(S)\to\Top$, $\ca{H}(s_1,s_2)=|S_{\geqslant s_1}|\times D(s_2)$.
\end{rem}

\begin{rem}
Note that, for the diagram $\ast$ which contains only singleton spaces (that is $\ast(s)=\pt$ for all $s\in S$), the homotopy colimit coincides with the geometric realization of the underlying poset: $\hocolim_S \ast=|S|$.
\end{rem}

\subsection{Completion along a diagram}\label{completion}

In case the values of the diagram $D$ are finite sets with discrete topology, the homotopy colimit can be easily defined by the following construction.

\begin{con}\label{conCompletion}
Let $S$ be a poset and $D\colon \cat(S)\to \FinSets$ an arbitrary diagram. Consider the set
\[
S^D=\bigsqcup\nolimits_{s\in S}D(s)
\]
and define a partial order on it: $d_1\leqslant d_2$ (for $d_1\in D(s_1)$, $d_2\in D(s_2)$) if and only if the following holds true
\[
\begin{cases}
  s_1\leqslant s_2 \\
  D(s_1\leqslant s_2)(d_1)=d_2.
\end{cases}
\]
The set $S^D$ with this partial order will be called \emph{the completion of the poset} $S$ along the diagram $D$. Note that there is a canonical monotonic map from the completion to the original set
\[
q_D\colon S^D\to S,
\]
which, for each $s\in S$, maps the elements $d\in D(s)\subset S^D$ to $s\in S$.
\end{con}

On an abstract level, the completion operation is taking a coend with values not in topological spaces, but in the category of posets.
Taking into consideration the constructions mentioned above, the next proposition easily follows.

\begin{prop}\label{proHocolim}
For any diagram $D\colon\cat(S)\to\FinSets$ of finite sets, there is a natural homeomorphism
\[
\hocolim_SD\cong |S^D|.
\]
\end{prop}

\begin{rem}
The completion operation is a partial order analogue of the simplicial replacement, known in homotopy theory~\cite{Dugger}.
\end{rem}

\subsection{Étale spaces}\label{etalespaces}

This paragraph describes how homotopy colimits from the previous subsection relate to the classical notion of étale space of a sheaf.

\begin{defin}
Let $\ca{F}\colon \cat(\Omega_X\setminus\{\varnothing\})^{op}\to \Vv$ be a presheaf on a topological space $X$. Then the étale space $F(\ca{F})$ of this presheaf is the topological space defined as follows.
\begin{enumerate}
    \item As a set,
    \[
    F(\ca{F}) = \bigsqcup\nolimits_{x\in X}\ca{F}_x;
    \]
    \item The base of topology on $F(\ca{F})$ is defined by the sets
    \[
    U_v = \{v_x\mid x\in U,\: (\ca{F}(U)\to \ca{F}_x)(v) = v_x\}
    \]
    for each open set $U$ and an element $v\in \ca{F}(U)$.
\end{enumerate}
\end{defin}

Recall that there is an equivalence between categories $\PrePos$ and $\AlexTop$, under which any poset $S$ corresponds to Alexandrov topological space $X_S$, see Proposition~\ref{propPosTop}.

\begin{prop}\label{etaleequiv}
    Let $S$ be a poset  and  $D\colon \cat(S)\to \FinSets$ be an arbitrary diagram. Let $\ca{D}$ be the presheaf on the topological space $X_S$ corresponding to this diagram. Then the étale space $F(\ca{D})$ is homeomorphic to $X_{S^D}$, where $S^{D}$ is the completion of the poset $S$ along the diagram $D$.
\end{prop}

\begin{proof}
It can be seen that $S^D$ and $F(\ca{D})$ coincide as sets, by definition. Consider an arbitrary upper order ideal $J$ in $S^D$ and check that this set is open in the topology on $F(\ca{D})$. Indeed, $J$ can be represented as a union of base sets:
\[
J = \bigcup\nolimits_{d\in J}\: (U_s)_d,
\]
where for each $d\in J$ we denote by $s\in S$ such element of the poset that $d\in D(s)$ and by $U_s$, as usual, the minimal open neighbourhood of $s$ in $S$.

Now consider an arbitrary open set $A$ in $F(\ca{D})$ and check that this set is an upper order ideal in terms of the order on $S^D$. Indeed, $A = \bigcup\nolimits U_t$ for some open sets $U\in X_S$ and elements $t\in \ca{D}(U)$. We take an arbitrary pair of elements in $F(\ca{D})$, $d_1,d_2\in A$ such that $d_1\leq d_2$. Then there are $U$ and $t\in \ca{D}(U)$ such that $d_1\in U_t$. This means that $d_1 = t_{s_1}$ for some $s_1\in U$. The inequality $d_1\leq d_2$ means that $s_1\leq s_2$ for $d_2\in D(s_2)$. Therefore, $s_2\in U$. Thus, there is a map $\ca{D}(U)\to \ca{D}_{s_2}$, which, by properties of sheaves, equals the composition of the map $\ca{D}(U)\to \ca{D}_{s_1}$ with the map $\ca{D}(s_1\leq s_2)$, and therefore this map sends $t$ to $d_2$, which means $d_2\in U_t$, that is $d_2\in A$.
\end{proof}

\begin{rem}
In the remaining part of the paper we use only the language of completions along a diagram to formulate main theorems (see Section~\ref{secInflation}). However, Proposition~\ref{etaleequiv} shows that all results can be reformulated in terms of étale spaces. We advise a curious reader to follow through the paper taking into consideration this equivalent language. It adds a topological perspective onto our combinatorial approach.
\end{rem}

\subsection{Simplicial posets}\label{subsecSimpPosets}

\begin{defin}\label{definSimpPoset}
A poset $S$ is called \emph{simplicial poset} if the following two conditions hold true
\begin{enumerate}
  \item There exists a unique least element $\hat{0}\in S$.
  \item For each $s\in S$, the subset $S_{\leqslant s}$ is isomorphic to a Boolean lattice.
\end{enumerate}
\end{defin}

The elements of a simplicial poset are called \emph{simplices}. The number $\rk s-1$, where $\rk s$ denotes the rank of the Boolean lattice $S_{\leqslant s}$, is called \emph{the dimension} of the simplex. We use the standard terminology: simplices of dimension $0$ are called \emph{vertices}, of dimension $1$ --- \emph{edges},  and  so on. From the existence of a rank function, it follows that simplicial posets are graded. The maximal dimension the simplices of $S$ is called the dimension of $S$. A simplicial poset is called \emph{pure} if all its maximal simplices have the same dimension.

\begin{rem}\label{remSimpPosetFromComplex}
The set of simplices of an arbitrary finite simplicial complex (including the empty simplex), ordered by inclusion, forms a simplicial poset.
\end{rem}

If $S$ is simplicial, then we call the set $S\setminus\{\hat{0}\}$ a \emph{geometric simplicial poset}. By definition, the geometric realization of a simplicial poset is the geometric realization of the corresponding geometric simplicial poset. This means the empty simplex is never considered in the geometric realization.

\begin{defin}
If $S$ is simplicial and  $s\in S$, then the poset $\link_Ss=S_{\geqslant s}$ is called the \emph{link} of the simplex $s$. A link is a simplicial poset.
\end{defin}

It follows that the geometric realization $|\link_Ss|$ is the geometric realization of the poset $S_{>s}=S_{\geqslant s}\setminus \{s\}$ (with the minimal element removed to make it geometric). Therefore we have $|\link_Ss|\cong |S_{>s}|$.

In the following we denote simplices of simplicial posets by letters $I,J$, etc. instead of $s,t$ as in general posets.

\subsection{Cohen--Macaulay simplicial posets}\label{subsecCMposets}

\begin{defin}
A simplicial poset $S$ is called \emph{homotopically Cohen–Macaulay} of dimention $n-1$ if the following conditions hold true
\begin{enumerate}
  \item $S$ is pure of dimension $n-1$;
  \item $|S|$ is homotopy equivalent to a wedge of $(n-1)$-dimensional spheres;
  \item For each (non-empty) simplex $I\in S$,  the geometric realization of its link, $|\link_SI|$, is homotopy equivalent to a wedge of $(n-2-\dim I)$-dimensional spheres.
\end{enumerate}
\end{defin}

\begin{rem}\label{remZeroWedge}
Recall the standard convention. A wedge of zero-many spheres (of any dimension) is assumed a contractible space. Vice versa, a contractible space is assumed a point, up to homotopy, and therefore it is a wedge of spheres (zero many wedge summands).
\end{rem}

\begin{rem}
It follows from the definition that, whenever $S$ is homotopically Cohen-Macaulay poset of dimension $n-1$, then $\link_SI$ is homotopically Cohen-Macaulay poset of dimension $n-2-\dim I$.
\end{rem}

\begin{rem}
We say that $S$ is homotopically Cohen--Macaulay geometric simplicial poset if it is obtained from homotopically Cohen--Macaulay simplicial poset by removing the least element. The necessity to remove this little fellow (the empty simplex) becomes really annoying since, in the paper, we sometimes reverse the order, and the least element becomes the greatest element. Anyhow, when we consider diagrams over simplicial posets, we never associate any meaningful set to the empty simplex: this is why we always have to remove it. Additional terminology, ``geometric simplicial poset'', was introduced for convenience of exposition.
\end{rem}

\section{The main construction}\label{secInflation}

\subsection{Inflation of a poset}

\begin{defin}\label{definInflation}
Let $P$ be a poset, and $D\colon \cat(P^*)\to \FinSets$ be a diagram on the dual poset $P^*$. In other words, to each element $I\in P$, a finite set $D(I)$ is assigned, and whenever $I\leqslant J$, there is a map $D(J\geqslant I)\colon D(J)\to D(I)$. \emph{Inflation} of $P$ along the diagram $D$ is the poset
\[
P_D=((P^*)^D)^*,
\]
that is the completion of $P^*$ along the diagram $D$ followed by reversal of the order.
\end{defin}

\begin{rem}\label{remMapQd}
From Construction~\ref{conCompletion} it follows immediately that there is a canonical monotonic map $q_D\colon P_D\to P$.
\end{rem}

We recall that a geometric simplicial poset is a poset that becomes simplicial after adding a formal empty simplex, the least element $\hat{0}$, see Subsection~\ref{subsecSimpPosets}.

\begin{prop}
If $P$ is a geometric simplicial poset, then $P_D$ is also a geometric simplicial poset.
\end{prop}

\begin{proof}
Consider an arbitrary element $d\in D(q)$, $q\in P$. Then we have
\[
(P_{D})_{\leqslant d} = \{t\mid \exists p\in P,\: t\in D(p),\: p\leq q,\: D(p\geq q)(d)=t\}.
\]
We note that for each $t\in (P_{D})_{\leqslant d}$ such $p\in P$ from the formula above is unique, and vice versa for each $p$ where $p\leq q$ there is a unique corresponding $t\in (P_{D})_{\leqslant d}$. Thus, there is a bijective correspondence between $P_{\leqslant q}$ and $(P_{D})_{\leqslant d}$, under which $q$ corresponds to $d$. It also can be easily checked that the poset structure does not change under this correspondence, which means it is an isomorphism. All lower cones of $P_D$ are therefore boolean lattices.
\end{proof}

\subsection{Example: inflated complexes of Wachs}

The following construction was introduced by Wachs in~\cite{Wachs}.

\begin{con}
Consider a finite simplicial complex $K$ with the vertex set $V=\{1,\ldots,m\}$ and a sequence of positive integers $(n_1,\ldots, n_m)$, one for each vertex of $K$. Consider the simplicial complex $K(n_1,\ldots,n_m)$ defined as follows
\begin{enumerate}
  \item The vertex set of $K(n_1,\ldots,n_m)$ is $V=\bigsqcup_{i\in[m]}[n_i]$. In other words, each vertex of $K$ is ``inflated'' into $n_i$ many copies. Let $p\colon V\to [m]$ denote the natural projection, which returns the index of summand in the disjoint union to which an element $v\in \bigsqcup_{i\in[m]}[n_i]$ belongs.
  \item A tuple $\{v_1,\ldots,v_s\}\subset V$ is a simplex of $K(n_1,\ldots,n_m)$ if and only if $p(v_1),\ldots,p(v_s)$ are pairwise distinct and the tuple $\{p(v_1),\ldots,p(v_s)\}$ is a simplex of $K$.
\end{enumerate}
The complex $K(n_1,\ldots,n_m)$ is called \emph{the vertex-inflation}\footnote{Originally it was called just \emph{inflation} of $K$. To distinguish it from the constructions of our paper, we give this term a more specific name.} of $K$.
\end{con}

This construction happens to be a particular case of inflations of simplicial posets along a diagram, introduced in Definition~\ref{definInflation}.

\begin{con}\label{conDiagramForVertexInflation}
As before, consider a simplicial complex $K$ (treated as geometric simplicial poset), and a sequence $(n_1,\ldots, n_m)$ of numbers corresponding to its vertices. Let $N_i$ be a finite set of cardinality $n_i$. Define the diagram $D_{\ver}\colon \cat(K^*)\to \FinSets$ as follows: each vertex $i$ of $K$ is mapped to $N_i$, while, for a general non-empty simplex $I\in K$, we define $D_{\ver}(I)$ by
\begin{equation}
    D_{\ver}(I) = \prod\nolimits_{i}N_i,\quad \text{over all vertices $i$ of $I$}.
\end{equation}
We also define the map $D_{\ver}(I\geq J)$ between $I$ and each its face $J$ as a natural projection (a face $J$ certainly contains a smaller number of vertices, so each factor $N_i$ that appears in $D_{\ver}(J)$ also appears in $D_{\ver}(I)$ and the projection is well defined).
\end{con}

The next statement is straightforward.

\begin{prop}\label{multiclique2}
The vertex-inflated simplicial complex $K(n_1,\ldots,n_m)$ coincides with the inflation $K_{D_{\ver}}$ of $K$ along the projection diagram $D_{\ver}$ described above.
\end{prop}

\begin{rem}
An open set $U$ in the Alexandrov topology corresponding to $K^*$ is an upper order ideal in $K^*$. It means that $U$ is a simplicial subcomplex in $K$ (see details in Remark~\ref{remOpenSubcomplexes} below). It is fairly straightforward to prove the following formula for the values of the sheaf $\ca{D}_{\ver}$ corresponding to the diagram $D_{\ver}$:
\begin{equation}\label{eqVertSheaf}
\ca{D}_{\ver}(U)=\prod\nolimits_{i}N_i,\quad \text{over all vertices $i$ of $U$}.
\end{equation}
\end{rem}

\begin{rem}\label{remVertexInflationFlabby}
Since all sets $N_i$ in Construction~\ref{conDiagramForVertexInflation} are non-empty it follows that the sheaf $\ca{D}_{\ver}$ corresponding to the diagram $D_{\ver}$ is inhabited. Flabbiness of the sheaf $\ca{D}_{\ver}$ follows from formula~\eqref{eqVertSheaf}. Indeed, all restriction maps of $\ca{D}_{\ver}$ are projections hence surjective.
\end{rem}

\subsection{Example: cliques in multigraphs}

We use the term ``multigraph'' for graphs having multiple edges, but not loops, see details in~\cite{AyzRukh}.

\begin{defin}\label{definMultigraph}
A multigraph is a $1$-dimensional simplicial poset.
\end{defin}

Any multigraph can be considered as a particular case of an inflation of a simple graph in the following sense. Consider a multigraph $G_\mu$ where $G=(V,E)$ is the underlying simple graph, and $\mu\colon E\to \Zo_{>0}$ is the multiplicity function. Consider the poset $P = \ca{F}(G)=V\sqcup E$ of faces (vertices and edges) of this graph, ordered by inclusion. Finally, consider the diagram $D_{\edge}\colon \cat(P^*)\to \FinSets$ that maps each vertex $v\in V$ of the graph to a singleton set and each edge $e\in E$ to the set of cardinality $\mu(e)$. The next statement is straightforward.

\begin{prop}\label{multigraph}
    For a multigraph $G_\mu$, the poset $P = \ca{F}(G)$ of faces of the corresponding simple graph, and the diagram $D_{\edge}\colon \cat(P^*)\to \FinSets$ we have
    \begin{equation}
        |P_D|\simeq |G_\mu|.
    \end{equation}
\end{prop}

A clique complex (also called Whitney complex) of a graph is the poset of all cliques of the corresponding graph --- which happens to be a simplicial poset, and, moreover, a simplicial complex. This notion is standard and well known in topological combinatorics~\cite[Def.9.1]{Kozlov}, computational topology~\cite{CriKol}, topological data analysis~\cite[Def.2.16]{DeiWang} and beyond. Analogously, \emph{the clique complex} of a multigraph is the poset of all cliques of a multigraph, see~\cite{AyzRukh}. Multiclique complex is a simplicial poset but may fail to be a simplicial complex.

Consider a multigraph $G_\mu$. We claim that its clique complex $mF(G_\mu)$ can be conceptually described as the inflation of clique complex $mF(G)$ of the corresponding simple graph $G$ along the appropriate diagram.

\begin{con}
Consider the clique complex $P = mF(G)$ of the underlying simple graph $G$. We define the diagram $D_{\edge}\colon \cat(P^*)\to \FinSets$ on the dual poset inductively. At first, we define the values on the elements of dimension $0$ and $1$ the same way as in the previous example, see Proposition~\ref{multigraph}: each vertex is mapped to a singleton, each edge is mapped to a set $M_e$ of cardinality $\mu(e)$. For an arbitrary clique $I\in P$ of the simple graph $G$, we define the value $D_{\edge}(I)$ by
\begin{equation}
    D_{\edge}(I) = \prod\nolimits_{e}M_e,\quad \text{over all edges $e$ in $I$}.
\end{equation}
We define the structure map $D_{\edge}(I\geq J)$ between $I$ and each its face $J$ as a natural projection (a face $J$ certainly contains a smaller number of edges, so each factor $M_e$ that appears in $D_{\edge}(J)$ also appears in $D_{\edge}(I)$ and the projection is well defined).
\end{con}
The next statement is almost tautological.

\begin{prop}\label{multiclique2}
    For a multigraph $G_\mu$, its multiclique complex $mF(G_\mu)$ is isomorphic to the inflation $P_{D_{\edge}}$ of the clique complex $P=mF(G)$ along the projection diagram $D_{\edge}$ described above.
\end{prop}

\begin{rem}
We have the following formula for the values of the sheaf $\ca{D}_{\edge}$ corresponding to the diagram $D_{\edge}$:
\begin{equation}\label{eqEdgeSheaf}
\ca{D}_{\edge}(U)=\prod\nolimits_{e}M_e,\quad \text{over all edges $e$ of $U$}.
\end{equation}
Here $U$ is a simplicial subcomplex of $P$.
\end{rem}

\begin{rem}\label{remEdgeInflationFlabby}
Since all sets $M_e$ in Construction~\ref{conDiagramForVertexInflation} are non-empty, the sheaf $\ca{D}_{\edge}$ is inhabited. It is also flabby as follows from formula~\eqref{eqEdgeSheaf} as all restriction maps of $\ca{D}_{\edge}$ are projections hence surjective.
\end{rem}

\section{The main results}\label{secFormulations}

Let $\ca{F}(\Delta^{n-1})$ denote the set of all non-empty faces of the $(n-1)$-dimensional simplex, ordered by inclusion.

\begin{thm}\label{thmSimplex}
Let $P=\ca{F}(\Delta^{n-1})$  and  $D\colon \cat(P^*)\to \FinSets$ be a diagram on the dual poset such that its corresponding sheaf $\ca{D}$ is inhabited  and  flabby. Then the geometric realization of the inflation $|P_D|$  is homotopy equivalent to a wedge of $(n-1)$-dimensional spheres.
\end{thm}

To put it briefly, an inflation of the $(n-1)$-dimensional simplex along a flabby and inhabited sheaf is homotopy equivalent to a wedge of $(n-1)$-dimensional spheres. We denote the number of spheres in this wedge by $n(D)$, which is, certainly, a non-negative integer.

\begin{thm}\label{thmGenSimpPoset}
Let $P$ be an arbitrary connected geometric simplicial poset of dimension $n-1$  and  $D\colon \cat(P^*)\to \FinSets$ be a diagram on the dual poset such that the corresponding sheaf $\ca{D}$ is inhabited and flabby. Then there is a homotopy wedge decomposition
\begin{equation}\label{eqMainWedgeDecomposition}
|P_D|\simeq |P|\vee\bigvee\nolimits_{I\in P}(\Sigma^{\dim I+1}|\link_PI|)^{\vee n(D_I)},
\end{equation}
where $D_I$ denotes the restriction of the diagram $D$ to the subset $(P_{\leqslant I})^*$.
\end{thm}

It is also useful to formulate a special (and fairly straightforward) corollary of the previous theorem as a separate result.

\begin{thm}\label{thmCM}
Let $P$ be a homotopically Cohen-Macaulay geometric simplicial poset of dimension $n-1$ and $D\colon \cat(P^*)\to \FinSets$ be a diagram such that the corresponding sheaf $\ca{D}$ is inhabited  and  flabby. Then $P_D$ is a homotopically Cohen-Macaulay simplicial poset of dimension $n-1$ as well.
\end{thm}

For a fixed dimension $n-1$, Theorem~\ref{thmGenSimpPoset} follows from Theorem~\ref{thmSimplex} by applying the poset fiber theorem of Bj\"{o}rner, Wachs  and  Welker~\cite{BjWW}. We present this theorem and the corresponding corollary in the next section. Theorem~\ref{thmCM} follows from Theorem~\ref{thmGenSimpPoset} by directly applying the definition of a homotopically Cohen-Macaulay simplicial poset.

The principal Theorem~\ref{thmSimplex} is proven by induction on $n$, however Theorem~\ref{thmCM} is applied in the step of induction. Overall, it seems methodologically justified to present all the theorems in a single section to keep better track of induction scheme. For convenience, we denote the statements of Theorem~\ref{thmSimplex}, \ref{thmGenSimpPoset} and \ref{thmCM} for a fixed $n$ by $\Th_1(n)$, $\Th_2(n)$  and  $\Th_3(n)$, correspondingly.

\section{Proofs}\label{secProofs}

\subsection{$\Th_1(\leqslant n)\Rightarrow\Th_2(n)$.}

This implication is based on the result of Bj\"{o}rner, Wachs and Welker, the poset fiber theorem~\cite{BjWW}. Let $f\colon S\to T$ be a morphism of posets. For each element $t\in T$ the subset $f^{-1}(T_{\leqslant t})$ (or its geometric realization $|f^{-1}(T_{\leqslant t})|$) is called the fiber of the map. Recall that a topological space $X$ is called $r$-connected if $\pi_i(X)=1$ for all $i\leqslant r$.

\begin{thm}[{The Poset Fiber Theorem; \cite[Thm.1.1]{BjWW}}]\label{thmFiberThm}
Let $f\colon S\to T$ be a morphism of posets such that, for each element $t\in T$, the fiber $|f^{-1}(T_{\leqslant t})|$ is $\dim |f^{-1}(T_{<t})|$-connected topological space. If $|T|$ is connected, then there is a homotopy equivalence
\begin{equation}\label{eqWedgeGeneral}
|S|\simeq |T|\vee\bigvee\nolimits_{t\in T} |f^{-1}(T_{\leqslant t})|\ast |T_{>t}|.
\end{equation}
If $|T|$ is not connected, then the formula applies to all connected components of $|T|$ separately.
\end{thm}

We use the poset fiber theorem to prove the implication $\Th_1(\leqslant n)\Rightarrow\Th_2(n)$.

\begin{proof}
Recall that $P$ is an arbitrary geometric simplicial poset of dimension $n-1$ in Theorem~\ref{thmGenSimpPoset} and $D\colon \cat(P^*)\to \FinSets$ is a diagram on the dual poset. The corresponding sheaf $\ca{D}$ is inhabited and flabby. Consider the natural projection map $q_D\colon P_D\to P$, as introduced in Remark~\ref{remMapQd}. We apply the poset fiber theorem (Theorem~\ref{thmFiberThm}) to this map of posets.

At first, we check that the assumptions of the Poset Fiber Theorem are satisfied. Since $P$ is a geometric simplicial poset, for each $I\in P$ the subset $P_{\leqslant I}$ is isomorphic to the poset $\ca{F}(\Delta^{k})$ of all non-empty faces of a simplex (that is a Boolean lattice without the least element),  and  $k\leqslant n-1$. Moreover, $P_{\leqslant I}$ is an exactly embedded subset in $P$. We note that the fiber $(q_D)^{-1}(P_{\leqslant I})$ coincides, by definition, with the inflation of the simplex $P_{\leqslant I}$ along the restriction $D|_{(P_{\leqslant I})^*}$ of the diagram $D$ to the subset $(P_{\leqslant I})^*\subset P^*$.

In Alexandrov topology corresponding to $P^*$, a subset $(P_{\leqslant I})^*=(P^*)_{\geqslant I}$ is always open (because it is an upper order ideal). Therefore, the restriction $\ca{D}_I=\ca{D}|_{P_{\leqslant I}}$ of the corresponding flabby sheaf $\ca{D}$ to this subset is flabby as well, see Remark~\ref{remFlabbyOnOpen}. Overall, all conditions of Theorem~\ref{thmSimplex} hold: there is a flabby sheaf on the (opposite) category of faces of the simplex of dimension $k\leqslant n-1$.

Since we assumed $\Th_1(\leqslant n)$ holds true, it follows from Theorem~\ref{thmSimplex} that the fiber $|q_D^{-1}(P_{\leqslant I})|$ is a $k$-dimensional space, homotopy equivalent to a wedge of $n(D_I)$ many $k$-dimensional spheres. Therefore, the fiber $|(q_D)^{-1}(P_{\leqslant I})|$ is a $(k-1)$-connected space. The principal assumption of the Poset Fiber Theorem holds true. Thus, by applying the Poset Fiber Theorem to the map $q_D\colon P_D\to P$, we get the homotopy equivalence
\begin{equation}
|P_D|\simeq |P|\vee\bigvee\nolimits_{I\in P} |q_D^{-1}(P_{\leqslant I})|\ast |P_{>I}|.
\end{equation}
Note that $P_{>I}$ is, by definition, a link of the $k$-dimensional simplex $I$ in the simplicial poset $P$, and that the space $|f^{-1}(P_{\leqslant I})|$, as we already have proven, is homotopy equivalent to the wedge $\bigvee_{n(D_I)}S^k$. Therefore,
\[
|f^{-1}(P_{\leqslant I})|\ast |P_{>I}|\simeq (\Sigma^{k+1}|\link_PI|)^{\vee n(D_I)}
\]
From this observation we conclude that the statement of $\Th_2(n)$ holds true.
\end{proof}

\subsection{$\Th_2(n)\Rightarrow\Th_3(n)$.}

This implication follows almost straightforwardly from the definition of a homotopically Cohen-Macaulay simplicial poset.

\begin{proof}
Let $P$ be a geometric simplicial homotopically Cohen-Macaulay poset of dimension $n-1$. Then $|P|$ is homotopy equivalent to a wedge of $(n-1)$-dimensional spheres,  and  for each (non-empty) simplex $I$ of dimension $k$ there is a homotopy equivalence $|\link_PI|\simeq \bigvee S^{n-2-k}$. We suppose the statement $\Th_2(n)$ holds, from which follows the homotopy equivalence
\begin{multline}
  |P_D|\simeq |P|\vee\bigvee\nolimits_{I\in P}(\Sigma^{\dim I+1}|\link_PI|)^{\vee n(D_I)}\simeq \\
  \simeq \left(\bigvee S^{n-1}\right)\vee\bigvee\nolimits_{I\in P}(\Sigma^{\dim I+1}\left(\bigvee S^{n-2-\dim I}\right)^{\vee n(D_I)})\simeq \bigvee S^{n-1}.
\end{multline}
So far, we proved that an inflation of a homotopically Cohen-Macaulay geometric simplicial poset is homotopy equivalent to a wedge of spheres. Now we need to check that all links of simplices of this inflation are as well homotopy equivalent to wedges of spheres. This statement is quite technical and can be proved similar to~\cite[Stm.~5.8]{AyzRukh}.
\end{proof}

\subsection{Proof of Theorem~\ref{thmSimplex}.}

The main Theorem~\ref{thmSimplex} is proved by induction on $n\geqslant 1$. This is where the property of flabbyness is finally been used.

\begin{rem}\label{remOpenSubcomplexes}
Let, as in Theorem~\ref{thmSimplex}, $P=\ca{F}(\Delta^{n-1})$ be a poset of non-empty faces of a simplex  and  $D$ be a diagram on the poset $P^*$. The corresponding sheaf is given on the topological space $X_{P^*}$. Open sets in this topology are upper order ideals in $P^*$ or, equivalently, lower order ideals in $P$. A curious reader can notice that lower order ideals in the poset of faces of a simplex on a set of vertices $[n]$ are simplicial complexes on the same set $[n]$. An interesting interpretation emerges: open sets are simplicial subcomplexes. For a simplex $I\in P$ its minimal open neighbourhood $U_I$ is a simplicial subcomplex generated by this simplex, that is the one that contains $I$ and  all its faces.

Construction of a sheaf $\ca{D}$ on the topology $X_{P^*}$ is equivalent to assigning a finite set $\ca{D}(U)$ to each simplicial subcomplex $U$ in $\Delta^{n-1}$  and  a restriction morphism $\ca{D}(U\supset V)\colon \ca{D}(U)\to \ca{D}(V)$ for each pair of nested subcomplexes $U\supset V$.
\end{rem}

Let us prove Theorem~\ref{thmSimplex}.

\subsection{$\Th_1(1)$.}

The base of induction is $n=1$. In this case, the simplex has dimension $n-1=0$, so the simplex is a point. A sheaf on a singleton $\{v\}$ contains only one finite set,  and  therefore we have $|P_D|=D(v)\simeq \bigvee_{|D(v)|-1}S^0$. Any finite set is homotopy equivalent (and even homeomorphic) to a wedge of $0$-dimensional spheres.

\subsection{$\Th_1(\leqslant n-1)\Rightarrow\Th_1(n)$.}

Let us prove the step of induction $\Th_1(\leqslant n-1)\Rightarrow \Th_1(n)$ for $n\geqslant 2$. As the sheaf $\ca{D}$ is inhabited by assumption, we have $D(I)\neq\varnothing$ for each non-empty face $I$ of the simplex, that is for each element of the poset $P=\ca{F}(\Delta^{n-1})$. Let us introduce a numerical characteristic of a sheaf (or the corresponding diagram):
\[
c(D)=\sum_{I\in P, I\neq \varnothing}\#D(I)=\#P_D,
\]
which we call \emph{the complexity of the diagram} $D$. From the condition of inhabitedness it follows that $c(D)\geqslant 2^n-1$, whereby equality is attained only if $\#D(I)=1$ for each face $I$ of the simplex.

We will use the inner induction on the complexity $c(D)$ of the diagram. The base of induction, that is the case $c(D)=2^n-1$, can be checked easily. In this case all stalks of a sheaf are singletons,  and  an inflation $P_D$ is naturally isomorphic to the initial poset $P$. Therefore $|P_D|\cong |P|\cong \Delta^{n-1}$ is a contractible space, hence considered a wedge of $(n-1)$-dimensional spheres, see Remark~\ref{remZeroWedge}.

Assume that $c=c(D)>2^n-1$ and the theorem is proven for all inhabited and flabby sheaves over the simplex $P=\ca{F}(\Delta^{n-1})$ of any complexity less than $c$. We consider a set $M\subseteq P$, that contains all simplices $I$ such that $\#D(I)>1$. Since $c>2^n-1$, the set $M$ is non-empty. We consider any minimal by inclusion element $I_0$ in $M$, that is a non-empty simplex $I_0$ with two properties: (1) $\#D(I_0)\geqslant 2$, (2) for any $J\subsetneq I_0$,  $\#D(J)=1$.

Let us represent the set $D(I_0)$ as a disjoint union of two non-empty subsets
\[
D(I_0)=A_1\sqcup A_2, \quad A_1,A_2\neq\varnothing.
\]
Since $\#D(I_0)\geqslant 2$ such representation exists. Now we consider a partition of the diagram $D$ into the union of two diagrams-skyscrapers on the sets $A_1$  and  $A_2$. More formally, we define two diagrams $D_1,D_2$ (subdiagrams in $D$) on $P$ by all their values:
\begin{equation}\label{eqSubDiag}
D_\varepsilon(J)=\begin{cases}
                   D(J\geqslant I_0)^{-1}(A_\varepsilon), & \mbox{if } J\geqslant I_0 \\
                   D(J), & \mbox{otherwise}.
                 \end{cases}
\end{equation}
for $\varepsilon=1$ and $2$. We recall that $D$ is a diagram on the dual poset $P^*$, and therefore all its morphisms map from bigger simplices to smaller ones. In simpler terms, we constructed a diagram $D_\varepsilon$ from those and only those elements of stalks of $D$ that can be restricted to~$A_\varepsilon$. We can think of it as an inclusion $D_1,D_2\subset D$ in the sense that there are natural transformations of functors $D_1,D_2\to D$, monomorphic on each stalk. In particular, from this observation it follows the existence of natural inclusions of inflated posets $P_{D_1}\subset P_D$ and $P_{D_2}\subset P_D$.

It can be seen that the complexity of two diagrams, $D_1$ and $D_2$, is strictly less than that of~$D$:
\[
c(D_1)<c(D),\quad c(D_2)<c(D),
\]
Now, let us formulate several auxiliary statements, that carry out the step of inner induction.

\begin{lem}\label{lemInhFlabby}
The sheaves $\ca{D}_1$ and $\ca{D}_2$, that correspond to diagrams $D_1$ and $D_2$, are inhabited and flabby.
\end{lem}

\begin{proof}
From the inhabitedness of $D$, it follows that all sets $D(J)$, for $J\in P$, are non-empty. The sets $A_1$, $A_2$ are non-empty by construction. All morphisms $D(J\geqslant I_0)$ are surjective, because they are special cases of the restriction morphisms in the sheaf $\ca{D}$, which is flabby. The inhabitedness of diagrams $D_1$, $D_2$ follows.

Let us prove the flabbiness of a sheaf $\ca{D}_\varepsilon$, $\varepsilon=1,2$. Let $U\supset V$ be two open sets in the Alexandrov topology $X_{P^*}$, i.e, according to Remark~\ref{remOpenSubcomplexes}, two simplicial subcomplexes. We need to prove that the map $\ca{D}_\varepsilon(U\supset V)\colon \ca{D}_\varepsilon(U)\to \ca{D}_\varepsilon(V)$ is surjective. Let us consider three cases.
\begin{enumerate}
  \item $I_0\notin U$. In this case the morphism $\ca{D}_\varepsilon(U\supset V)\colon \ca{D}_\varepsilon(U)\to \ca{D}_\varepsilon(V)$ literally coincides with the restriction morphism $\ca{D}(U\supset V)\colon \ca{D}(U)\to \ca{D}(V)$ of the original sheaf which is surjective by assumption.
  \item $I_0\in U$, $I_0\notin V$. Let $u\in \ca{D}_\varepsilon(V)=\ca{D}(V)$ be an arbitrary section on $V$. We consider an arbitrary section $v\in A_\varepsilon\subset D(I_0)=\ca{D}(U_{I_0})$. We recall that $U_{I_0}$ is a minimal open neighbourhood of $I_0$ in the Alexandrov topology and, at the same time, a simplicial subcomplex, generated by the simplex $I_0$, see Remark~\ref{remOpenSubcomplexes}. The choice of $I_0$ implies that, for each proper face $J<I_0$, the set $D(J)$ is a singleton. Therefore, the conditions of the Gluing Lemma~\ref{lemGlue} are satisfied for $U_{I_0}$ and $V$, which means there exists a section $u'\in \ca{D}(V\cup U_{I_0})$ on a subcomplex, generated by gluing a simplex $I_0$ to $V$. From the flabbiness of $\ca{D}$, it follows that there is a section $u''\in \ca{D}(U)$ extending $u'$ to the subcomplex $U$. We claim that $u''$ belongs to the subset of sections $\ca{D}_\varepsilon(U)\subset \ca{D}(U)$. This is practically a tautology: from the construction it follows that the restriction of $u''$ on the subcomplex $U_{I_0}$ is $v\in A_\varepsilon$, which actually means that the section $u''$ belongs to the specified sheaf $\ca{D}_\varepsilon$. So far, we proved that for any element $u\in \ca{D}_\varepsilon(V)$ there exists a covering element of it, namely $u''\in\ca{D}_\varepsilon(U)$, and therefore the map $\ca{D}_\varepsilon(U\supset V)$ is surjective.
  \item $I_0\in V\subset U$. This case can be proved analogously to the previous one, but the gluing procedure of $I_0$ to $V$ is unnecessary in this case.
\end{enumerate}
\end{proof}

\begin{lem}\label{lemUnion}
$P_D=P_{D_1}\cup P_{D_2}$.
\end{lem}

\begin{proof}
This statement is fairly straightforward since any element of a stalk of $D$, restricted to the simplex $I_0$, belongs either to $A_1$ or to $A_2$. The elements of stalks over the simplices that do not contain $I_0$, belong to both inflations.
\end{proof}

We denote by $D_{12}$ the intersection of ``subdiagrams'' $D_1$ and $D_2$ inside the bigger diagram $D$. We note that if $J\geqslant I_0$, then $D_{12}(J)=D_1(J)\cap D_2(J)=\varnothing$, because the corresponding preimages do not intersect (see the formula~\eqref{eqSubDiag}). On the other hand, for any $J$ such that $J\ngeqslant I$, both sets, $D_1(J)$ and $D_2(J)$, coincide with $D(J)$, and therefore coincide with their intersection. The following construction a general type of simplicial complexes which can arise as the support of $D_{12}$.

\begin{con}\label{conMainSubcomplex}
Let $I\in P=\ca{F}(\Delta^{n-1})$ be a face of the simplex on the vertex set $[n]$. Let us consider its simplicial subcomplex $K(I)$ that is a subset of the poset $P$ of the form
\[
K(I)=\{J\in P\mid J\ngeqslant I\}.
\]
There is a number of alternative ways to define this simplicial subcomplex. For example, $K(I)$ is a simplicial complex such that $I$ is its unique minimal non-simplex. Alternatively, $K(I)$ is the union of hyperfaces $\ca{F}_i=[n]\setminus \{i\}$ over all $i\in I$. Yet another description: $K(I)$ is the simplicial join of the boundary of the simplex $\dd U_I$ with the simplex $U_{[n]\setminus I}$. We recall that $U_J$ denotes the subcomplex generated by the simplex $J$, and that is consistent with the general terminology of minimal open neighbourhood (see Remark~\ref{remMinNbhdIsACone} and Remark~\ref{remOpenSubcomplexes}).
\end{con}

Let us describe the topological properties of $K(I)$.

\begin{lem}\label{lemCMofMain}
Let $P=\ca{F}(\Delta^{n-1})$ be the poset of faces of the $(n-1)$-dimensional simplex and $I\in P$ be its non-empty face. Then the simplicial complex $K(I)$ is homotopically Cohen–Macaulay of dimension $n-2$.
\end{lem}

\begin{proof}
The estimated dimension equals $n-2$ because $K(I)$ is a union of a non-empty set of hyperfaces of the simplex $\Delta^{n-1}$, see Construction~\ref{conMainSubcomplex}. The property of being homotopically Cohen–Macaulay can be proven in three steps. (1) A simplex is always a homotopically Cohen-Macaulay complex, (2) a boundary of a simplex is always a homotopically Cohen-Macaulay complex as well, (3) the join of homotopically Cohen-Macaulay complexes is homotopically Cohen-Macaulay as well. All 3 steps are straightforward and left as an exercise for the reader. Our last observation is that $K(I)=\dd U_{I}\ast U_{[n]\setminus I}$, according to Construction~\ref{conMainSubcomplex}, which finishes the proof.
\end{proof}

Returning back to the intersection $D_{12}$ of the subdiagrams, we have
\begin{equation}\label{eqIntersectDiagrams}
D_{12}=D_1|_{K(I_0)}=D_2|_{K(I_0)}=D|_{K(I_0)}.
\end{equation}

From the construction of $D_{12}$ and from the definition of the poset inflation it follows
\begin{equation}\label{eqPosetIntersection}
P_{D_{12}}=P_{D_1}\cap P_{D_2}.
\end{equation}

\begin{lem}\label{lemIntersectionFlabby}
The sheaf $\ca{D}_{12}$, which corresponds to the diagram $D_{12}$ on $K(I_0)$, is inhabited and flabby.
\end{lem}

\begin{proof}
The inhabitedness is straightforward, since the simplicial complex $K(I_0)$ is, by construction, a support of $D_{12}$ over the simplex $P=\ca{F}(\Delta^{n-1})$. The flabbiness follows from the formula~\eqref{eqIntersectDiagrams}, because $\ca{D}$ is flabby (by the assumption), $K(I_0)$ is an open set in the topology of $X_{P^*}$ (see Remark~\ref{remOpenSubcomplexes}), and a restriction of a flabby sheaf to an open set is a flabby sheaf as well (see Remark~\ref{remRestrictionOfSheaf}).
\end{proof}

Now we formulate the last useful technical lemma to gather all previous results together and finally prove Theorem~\ref{thmSimplex}.

\begin{lem}\label{lemInductionHomotopy}
Let $X_1,X_2$ be two cell subcomplexes of a CW-complex $Y$ such that $Y=X_1\cup X_2$ and two conditions hold: (1) $X_1$ is homotopy equivalent to a wedge of $k$-dimensional spheres,  $X_2$ is homotopy equivalent to a wedge of $k$-dimensional spheres; (2) the intersection $A=X_1\cap X_2$ is homotopy equivalent to a wedge of $(k-1)$-dimensional spheres. Then their union $Y$ is homotopy equivalent to a wedge of $k$-dimensional spheres.
\end{lem}

This statement is quite standard and easy to prove. A sufficiently detailed proof can be found in~\cite[Lm.4.1]{AyzRukh}.

\begin{proof}[Proof of Theorem~\ref{thmSimplex}]
At this point, we are ready to prove the step of the inner induction in Theorem~\ref{thmSimplex}. Consider the geometric realizations $|P_D|$, $|P_{D_1}|$, $|P_{D_2}|$ and $|P_{D_{12}}|$. They all are CW-complexes, and besides,
\[
|P_D|=|P_{D_1}|\cup |P_{D_2}|, \qquad |P_{D_{12}}|=|P_{D_1}|\cap |P_{D_2}|
\]
according to Lemma~\ref{lemUnion} and equality~\eqref{eqIntersectDiagrams}. The sheaves $\ca{D}_1$, $\ca{D}_2$ are inhabited and flabby over a simplex, and their complexity is less than the complexity of $\ca{D}$. Therefore the spaces $|P_{D_1}|$ and $|P_{D_2}|$ are homotopy equivalent to wedges of $(n-1)$-dimensional spheres, by the hypothesis of the inner induction. The space $|P_{D_{12}}|$ is the geometric realization of the inflation of simplicial poset $K(I_0)$ along the diagram $D_{12}$. The corresponding sheaf is inhabited and flabby, according to Lemma~\ref{lemIntersectionFlabby}. In addition, $K(I_0)$ is the Cohen-Macaulay simplicial poset of dimension $n-2$. Therefore, we can apply the statement $\Th_3(n-1)$ (which is already proven as follows from the induction hypothesis). The statement $\Th_3(n-1)$ implies that the space $|K(I_0)_{D_{12}}|$ is homotopy equivalent to a wedge of $(n-2)$-dimensional spheres. Now we are in position to apply Lemma~\ref{lemInductionHomotopy}: it implies that $|P_D|$ is homotopy equivalent to a wedge of $(n-1)$-dimensional spheres. The proof is complete.
\end{proof}

\section{Final remarks}\label{secFinalRemarks}

\subsection{Nondegenerate simplicial maps}
Theorem~\ref{thmMainIntro} stated in the introduction follows easily from results of Section~\ref{secFormulations}. Indeed, if $f\colon N\to K$ is a surjective simplicial map, then the corresponding sheaf $\ca{D}_f$ over the poset $(K\setminus\{\varnothing\})^*$ (see Construction~\ref{conDiagramFromMap}) is inhabited. It can be seen that the inflation $(K\setminus\{\varnothing\})_{D_f}$ is isomorphic to the original covering space $N$, therefore the three theorems of Section~\ref{secFormulations} imply the three items of Theorem~\ref{thmMainIntro} in a straightforward manner.

\subsection{Vertex inflations}
As discussed in Section~\ref{secInflation}, there are two important classes of examples of inflated posets: vertex-inflations introduced by Wachs and edge-inflations, leading to mutliclique complexes of multigraphs. According to Remarks~\ref{remVertexInflationFlabby} and~\ref{remEdgeInflationFlabby}, both classes correspond to flabby inhabited sheaves, therefore Theorems~\ref{thmSimplex}, \ref{thmGenSimpPoset}, and~\ref{thmCM} are applicable.

Speaking of vertex-inflated simplicial complexes, the theorems proved in the previous paragraph imply the following result, originally discovered by Wachs~\cite{Wachs}.

\begin{cor}
The following hold true
\begin{enumerate}
  \item For a simplex $\Delta^{n-1}$, its vertex-inflation $\Delta^{n-1}(k_1,\ldots,k_n)$ is homotopy equivalent to a wedge of $n_I=\prod_{i\in I}(k_i-1)$ many $(n-1)$-dimensional spheres, where $I=[n]$.
  \item For arbitrary connected simplicial complex $K$, its vertex-inflation $K(k_1,\ldots,k_m)$ has homotopy wedge decomposition
  \[
  |K(k_1,\ldots,k_m)|\simeq |K|\vee \bigvee_{I\in K\setminus\{\varnothing\}} (\Sigma^{\dim I+1}|\link_KI|)^{\vee n_I}.
  \]
  \item If $K$ is homotopically Cohen--Macaulay complex of dimension $n-1$, then so is its vertex inflation $K(k_1,\ldots,k_m)$.
\end{enumerate}
\end{cor}

Notice that item 1 in this case if fairly simple. The complex $\Delta^{n-1}(k_1,\ldots,k_n)$ is the simplicial join $[k_1]\ast\cdots\ast[k_n]$ of 0-dimensional complexes, and its homotopy type is a straightforward exercise in homotopy topology.

\subsection{Cliques in multigraphs and edge inflations}
Speaking of edge-inflated simplicial complexes, the theorems of the previous paragraph imply the following result, that originally appeared in~\cite{AyzRukh}.

\begin{cor}
Let $G_\mu$ be a multigraph with the underlying simple graph $G$ and multiplicity function $\mu$. Let $P=mF(G)$ be the clique complex of $G$ and $P_{D_{\edge}}=mF(G_\mu)$ --- the multiclique complex of the multigraph $G_\mu$. The following hold true
\begin{enumerate}
  \item For a full graph $K_n=\left([n],{[n]\choose 2}\right)$ and arbitrary edge multiplicity function $\mu\colon {[n]\choose 2}\to \Zo_{>0}$, the clique complex $mF((K_n)_\mu)$ of the corresponding multigraph is homotopy equivalent to a wedge of $(n-1)$-dimensional spheres.
  \item For arbitrary multigraph $G_\mu$, its clique complex $mF(G_\mu)$ has a homotopy wedge decomposition in terms of the clique complex $mF(G)$, and links of its simplices.
  \item If the clique complex $mF(G)$ of a graph is homotopically Cohen--Macaulay complex of dimension $n-1$, then so is the clique complex $mF(G_\mu)$ of multigraph $G_\mu$ supported on $G$.
\end{enumerate}
\end{cor}

The paper~\cite{AyzRukh} introduces more general concept of inflation along simplices of arbitrary dimensions and proves an analogous theorem for such general inflations. This generalization is also covered by Theorems~\ref{thmSimplex},\ref{thmGenSimpPoset},\ref{thmCM} of the present paper as the sheaves corresponding to these general consecutive simplex-inflations are flabby as well.

\begin{rem}
In~\cite{AyzRukh}, a functorial version of homotopy wedge decomposition was formulated and proved. We believe that the language of sheaves proposed in the present paper gives a more transparent and straightforward way to prove such functoriality results. Hopefully, the details will be elaborated in a separate paper.
\end{rem}

\end{document}